\documentclass[12pt,a4paper,notitlepage]{amsart}
\usepackage[latin1]{inputenc}
\usepackage[french,english]{babel}

\usepackage{amsmath,amssymb,amsthm}
\usepackage{xcolor}

\newcommand{\upperRomannumeral}[1]{\uppercase\expandafter{\romannumeral#1}}

\usepackage{cite}

\renewenvironment{abstract}[1]
  {\bigskip\selectlanguage{#1}%
   \begin{center}\bfseries\abstractname\end{center}}
  {\par\bigskip}

\theoremstyle{plain}
\newtheorem{theorem}{Theorem}[section]
\newtheorem{lemma}[theorem]{Lemma}
\newtheorem{corollary}[theorem]{Corollary}

\theoremstyle{definition}

\newtheorem{conjecture}[theorem]{Conjecture}

\newtheorem{remark}[theorem]{Remark}

\begin{document}

\title[Distribution of factorials modulo $p$]
{\bf Distribution of factorials modulo $p$}
\author{Oleksiy Klurman} 
\address{D\'epartment de Math\'ematiques et de Statistique,
Universit\'e de Montr\'eal, CP 6128 succ. Centre-Ville, Montr\'eal QC H3C 3J7, Canada
Canada} \email{\texttt{lklurman@gmail.com}}
\author{Marc Munsch}
\address{CRM, Universit\'e de Montr\'eal, 5357 Montr\'eal, Qu\'ebec }
\email{munsch@dms.umontreal.ca}
\date{\today}


\maketitle

\begin{abstract}{french}
 \small{On démontre que la suite $n!\,(\bmod\,p)$ prend au moins $\sqrt{\frac{3}{2}N}$ valeurs distinctes lorsque $n$ parcourt l'intervalle court $H\le n \le H+N$ où $N\gg p^{\frac{1}{4}}$, ceci constituant une amélioration de la borne triviale précédemment connue $\sqrt{N}.$ On explore le problème complémentaire des valeurs non atteintes par cette suite. Dans ce sens, on obtient, en moyenne sur les nombres premiers $p\le x$, une minoration du nombre de classes modulo $p$ évitées par la suite $n!\,(\bmod\,p)$.}
\end{abstract}

\begin{abstract}{english}
\small{We prove that the sequence $n!\,(\bmod\,p)$ occupies at least $\sqrt{\frac{3}{2}N}$ residue classes in the short interval $H\le n \le H+N$ and $N\gg p^{\frac{1}{4}}$ improving previously known trivial bound $\sqrt{N}.$ In the other direction, we estimate the average number of residue classes missed by the sequence $n!\,(\bmod\,p)$ for $p\le x.$}
\end{abstract}

\section{Introduction}

Following~\cite{TrAMS}, for each odd prime $p$ let $V(H,N)$ denote the number of distinct residue classes modulo $p$ that are taken by the sequence $\{n!, n=2,3\dots, p-1\},$ $H\le n\le H+N.$ Very little seems to be known about the behaviour of $V(H,N).$ P. Erd\"os conjectured that $2!,3!\dots, (p-1)!$ cannot be all distinct modulo $p$, in other words $V(0,p-1)\neq p-2$. Although the conjecture is widely open, B. Rokowska and A. Schinzel~\cite{RokoSchinzel}  proved that this strong condition implies some restrictions on the values of $p.$ This allows to verify that the conjecture holds true for $p<10^9$ (see \cite{Trudgian}). 
More generally, the following asymptotic is conjectured in \cite{Guy}: $$V(0,p-1)\sim \left(1-\frac{1}{e}\right)p.$$ 

In \cite{Zaha}, C. Cobeli, M. V\^aj\^aitu and A. Zaharescu provide a  strong support towards this conjecture  (see also \cite{Broughan}). They proved that for a random permutation $\sigma$ of the set $\{1,\cdots,p-1\}$, the products $$\left\{\displaystyle{\prod_{i=1}^{n}\sigma(i), n=1,\cdots,p-1}\right\}$$ cover the expected number of residue classes. This implies in particular that in case the sequence $\{n!, n=2,3\dots, p-1\}$ did not satisfy Guy's observation then it would not, in some sense, be a "standard" sequence amongst the set of all sequences of length $p.$

In a series of papers, \cite{TrAMS}, \cite{Garaev-Luca} and \cite{Crelle}, M. Garaev, F. Luca and I.  Shparlinski initiated extensive study of distribution properties of $n!(\bmod\ p).$
In particular, in \cite{TrAMS} the authors remark that the only known lower bound for 
$V(H,N)$ is a trivial one, namely $$V(H,N)\ge \sqrt{N-1}.$$
Indeed, this immediately follows from the fact that the remainders $\frac{n!}{(n-1)!}=n$ are all distinct for $1\le n\le p-1.$   

Motivated by this question, V. Lev considered a similar problem in every finite abelian group $G$. He showed \cite[Theorem 2]{Lev} that there is a permutation $(g_1,\cdots,g_{|G|})$ of the elements of $G$ such that the number of distinct sums of the form $$g_1+\cdots+g_j \,(1\leq j\leq |G|)$$ 
is $O(\sqrt{|G|})$ and noticed that this is the smallest size possible. By fixing a primitive root $g$ modulo $p$ and passing to indices with respect to $g,$ the question about the distribution of factorials reduces to considering the cyclic group of order $p-1$ and the permutation given by the indices. V. Lev observes that the conclusion one can draw concerning the sequence $n!$ is of a negative sort: in order to improve the lower bound on $V(H,N),$ combinatorics is not sufficient and one has to exploit the special properties of this particular sequence.

In this note, among other things, we present an elementary way to obtain nontrivial lower bound on $V(H,N)$ for all $N\gg p^{\frac{1}{4}+\varepsilon}.$

It is worth mentioning that nontrivial lower bounds for $V(0,p-1)$ were previously known. For instance, Theorem 13 from~\cite{TrAMS} immediately implies that 
\[ V(0,p-1)\ge \sqrt{\frac{5}{4}p}.\]
The latter was subsequently improved by Chen and Dai (see~\cite{ChenDai}) to exactly the same constant $\sqrt{\frac{3}{2}}.$ Despite being extremely short, the proof in~\cite{ChenDai} uses a deep theorem of Zhang on his solution of the Lehmer problem and does not generalize to the short intervals $H\le n\le N+H.$  

In the other direction, it was proved in \cite{Turkish} that there exists infinitely many primes $p$ such that  $n! (\text{mod}\  p)$ omits at least
\begin{equation}\label{turk1}
 p-V(0,p-1)\gg\frac{\log\log p}{\log\log\log p}
 \end{equation}
 residue classes. 
 Applying the method of \cite{Turkish} and replacing the unconditional error term in Chebotarev's theorem by the GRH error term, yields infinitely many primes such that
 \begin{equation}\label{turk2}
p-V(0,p-1)\gg \frac{\log p}{\log \log p}\cdot 
\end{equation} 
 We remark that in~\cite{Turkish}, due to the use of the bound on the least prime ideal from \cite{least}, one gets an extremely sparse set of primes satisfying~\eqref{turk1} and~\eqref{turk2}. We are going to prove an average analog of this result. 
 In fact, under the assumption of the Generalized Riemann Hypothesis (GRH), our result immediately implies existence of infinitely many primes with  
\[ p-V(0,p-1)\gg \frac{p^{\frac{1}{4}}}{\log p}\cdot\]

\vspace{3mm}

\section{Estimate for $V(H,N)$}

We are going to prove the following theorem:
\begin{theorem}\label{main}
The set of $n!\,(\bmod\,p),$ $H\le n\le H+N$ contains at least $\sqrt{\frac{3}{2}N}$ values for all $N\gg p^{\frac{1}{4}+\varepsilon}.$
\end{theorem}
To prove this theorem we begin with the following three lemmas.

\begin{lemma}\label{Burgess}
Let  $\epsilon>0$ be fixed and $P(x)=x^2+bx+c \in \mathbb{Z}[x].$  Then for all $H>0$ and sufficiently large prime $p$ and $N$ such that $N\gg p^{\frac{1}{4}+\epsilon}$, there exists $\delta>0$ such that

$$\#\left\{ y=P(x)\, (\bmod \,p),\,\, H\leq y \leq H+N\right\}=\frac{N}{2} + O(N^{1-\delta}). $$

\end{lemma}

\begin{proof} Multiplying both sides of $y=P(x)$ by $4$ and applying linear change of variables $x\rightarrow 2x+\alpha$, the claim boils down to counting quadratic residues in an interval of length $N.$ The result follows from Burgess bound, see \cite{Burgesssaving}. \end{proof}

\begin{lemma}\label{sieve}
Suppose $S\subset [H,H+N]$ and $|S|=\alpha N.$ Then there exists $d\le \frac{1}{\alpha}$ such that there are at least $\ge \frac{\alpha^3}{2}N$ solutions of the equation
\[a-b=d,\]
where $a,b\in S.$
\end{lemma}
\begin{proof}
For $k=\lceil{\frac{1}{\alpha}\rceil}+1,$ consider the following shifted sets $$S+1, S+2\dots S+k$$
By inclusion-exclusion formula we have 
\[N+k\ge|\cup_{i=1}^{k}(S+i) |\ge \sum_{i=1}^k|S+i|-\sum_{j<i} |S+i\cap S+j|=k\cdot\alpha N-\sum_{j<i}|S+i\cap S+j|.\]
Therefore,
$$\max_{i<j}|S+i\cap S+j|\ge 2\cdot \frac{k\cdot\alpha N-N-k}{k^2}\ge \alpha^3N.$$
This observation finishes the proof.

\end{proof}

\begin{lemma}\label{factopoly}
Suppose $P\in\mathbb{Z}[x].$ Then the equation
\begin{equation}\label{factorial}
(n!)^2=P(n)\,(\bmod\,p)
\end{equation}
has at most $\ll N^{3/4}$ solutions in $\mathbb{F}_p$ such that $H\le n\le H+N.$
\end{lemma}
\begin{proof}
Suppose \eqref{factorial} has $\alpha N $ solutions in the interval $[H,H+N].$ By Lemma \ref{sieve} there exists $d\le\frac{1}{\alpha}$ such that
\[(n!)^2=P(n)\,(\bmod\,p)\] and
\[((n+d)!)^2=P(n+d)\,(\bmod\,p)\]
hold for at least $\ge \frac{\alpha^3N}{2}$ values of $H\le n\le H+N.$
Subtracting last two equations we arrive at
\[(n!)^2(\prod_{k=1}^d(n+k)^2-1)\equiv P(n+d)-P(n)\,(\bmod\,p).\]
Combining this with~\eqref{factorial} we get the following polynomial congruence:
\[P(n)(\prod_{k=1}^d(n+k)^2-1)\equiv P(n+d)-P(n)\,(\bmod\,p).\]
This congruence, being non degenerate, has at most $2d+\deg P$ solutions. Thus,
\[\frac{1}{\alpha}\gg \alpha^3 N\]
and $\alpha\ll N^{-1/4}.$ This concludes the proof of the lemma. 
\end{proof}
We are now ready to prove the main result.
%\begin{theorem}\label{main}
%The set of $n!\,(\bmod\,p)$ contains at least $\sqrt{\frac{3}{2}p}$ values.
%\end{theorem}

{\bf Proof of Theorem \ref{main}.}
%\begin{proof}
Colour the set $[H,H+N]$ into $k$ colours such that $a$ and $b$ are coloured in the same way if and only if $a!\equiv b!\,(\bmod\,p).$
Observe that pairs $(n,n+1),$ $H\le n\le H+N$ are coloured in a different way. Indeed, the following two conditions $n!\equiv m!\,(\bmod\,p)$ and $(n+1)!\equiv (m+1)!\,(\bmod\,p)$ imply that $n=m.$

Consider now pairs of the form $(n,n+2).$ Suppose that $(n,n+2)$ and $(m,m+2)$ are coloured in the same way. This leads to 
\[n!=m!\,(\bmod\,p)\] and
\[(n+2)!=(m+2)!\,(\bmod\,p).\] 
Hence we must have
\[(n+1)(n+2)\equiv (m+1)(m+2)\,(\bmod\,p)\]
or, assuming that $n$ and $m$ are distinct 
\[n+m+3\equiv 0\,(\bmod\,p).\]
Since $1\le m,n\le p-1$ we have that $m+n+3=p$ or $m+n+3=2p.$ Latter implies that $m=p-1$ and $n=p-2$ or vice versa. But 
$(p-1)!\ne (p-2)!\,(\bmod\,p).$ So we are left to consider the first case. Let $m=p-3-n.$ Then
\begin{equation}\label{wilson}
n!\equiv (p-3-n)!\,(\bmod\,p).
\end{equation}
Multiplying \eqref{wilson} by 
$\displaystyle{\prod_{k=1}^{n+2}(p-3-n+k)}$ and using Wilson's theorem we end up with
\[n!\prod_{k=1}^{n+2}(p-3-n+k)\equiv -1\,(\bmod\,p).\] 
 Reducing both sides modulo $p$ we get
 \[(-1)^nn!(n+2)!\equiv -1\,(\bmod\,p).\]
 Multiplying both sides by the product $f(n+2)=(n+1)(n+2)$ leads to the equation 
 \[((n+2)!)^2\equiv (-1)^{n-1}f(n+2)\,(\bmod\,p).\]
 By Lemma \ref{factopoly} each of the equations $((n+2)!)^2\equiv f(n+2)\,(\bmod\,p)$ and \\ $((n+2)!)^2\equiv -f(n+2)\,(\bmod\,p)$ has at most $\ll N^{3/4}$ solutions in the interval $H\le n\le H+N.$ 
 
 Latter implies that we have at least $N-2+O(N^{3/4})$ pairs of the form $(n,n+2)$ that are coloured differently. We now suppose that pairs $(n,n+1)$ and $(m,m+2)$ are coloured in the same way.
 Then,
\[n!=m!\,(\bmod\,p)\] and
\[(n+1)!=(m+2)!\,(\bmod\,p).\]   
This implies 
\[n\equiv m^2+3m+1\,(\bmod\,p).\]
Therefore, using Lemma \ref{Burgess} we deduce that at least $ N/2+O(N^{1-\delta})$ pairs of the form $(n,n+1)$ do not correspond to pairs of the form $(n,n+2)$. Summarizing all the above we get at least $\frac{3}{2}N+ O(N^{1-\delta})$ pairs that are coloured in a different way and thus
\[k^2\ge \frac{3}{2}N+ O(N^{1-\delta}).\]
 
 %\end{proof}
\begin{remark} One may try to improve constant $\sqrt{\frac{3}{2}}$ in the last theorem by considering the pairs $(n,n+k)$ for larger values of $k\ge 3.$ Following the same lines as above, one arrives at the study of simultaneous equations of the form $m!\equiv n!\ (\bmod\ p)$ and
\[f(m,n)\equiv 0\ (\bmod\ p)\]
where $f$ is a polynomial $f(x,y)\in\mathbb{Z}(x,y)$ of degree $k-1.$ Unfortunately, we were unable to employ this strategy to achieve any further improvement. 
\end{remark}
\section{Upper bound for $V(0,p-1)$ on average}

As was mentioned in the introduction, it was proved in \cite{Turkish} that there exists infinitely many primes $p$ such that  $n!\,(\text{mod}\  p)$ omits at least
\begin{equation}\label{turk3}
 p-V(0,p-1)\gg\frac{\log\log p}{\log\log\log p}
 \end{equation}
 residue classes. In this section, we show that the number of "missing" residue classes tends to infinity on average. 

We fix a few notations. Let $L/\mathbb{Q}$ be a finite extension of degree $n_L.$ For any ideal $\mathfrak{I}$ of the ring of integers $\mathcal{O}_L$, we denote the norm of an ideal by $N_{L/\mathbb{Q}}(\mathfrak{I})$ and write $N(\mathfrak{I})$. We also denote by $f\left(\mathfrak{p}/p\right)$ the inertial degree $\vert\left[\mathcal{O}_L/\mathfrak{p}:\mathbb{F}_{p}\right]\vert$ of the ideal $\mathfrak{p}$ above the rational prime $p$. The function $\pi_L(x)$ will count the number of prime ideals $\mathfrak{p}$ such that $N(\mathfrak{p}) \leq x$. Finally, denote  by $d_L$ the absolute discriminant of $L$.
\begin{theorem}\label{average}
We have 
 
 $$\frac{1}{\pi(x)}\sum_{p\leq x} \left(p-V(0,p-1)\right) \gg \frac{\log\log x}{\log \log \log x}\cdot$$
 \end{theorem}
 
\begin{proof}
 Let $N$ be a parameter which will be determined later. For $n\ge 1$ we consider the family of polynomials
 $$f_n(t)=t(t+1)\dots(t+n-1)-1.$$
 It is well-known (see \cite[$9$, part $\romannumeral 8 $, chapter $2$, section $3$, Pb $121$]{Polya} that $f_n(t)$ is irreducible over $\mathbb{Q}$ for all $n\ge 1.$ 
 Let  $\rho_n(p)$ denote the number of roots of $f_n(t)$ modulo $p.$ We observe that $f_n(t_0)\equiv 0\,(\bmod\,p)$ implies
 $$(t_0+n-1)!=(t_0-1)!\,(\bmod\,p).$$ Therefore, each distinct root of $f_n(t)$ modulo $p$ increases the number of "missing" values by $1.$ We thus want to produce a lot of roots of $f_n$ for many values of $n.$

 Let $K_n=\mathbb{Q}(\alpha)$ be the extension of $\mathbb{Q}$ obtained by adjoining a root of $f_n.$ By $C_k$ we denote  the subset of all non ramified primes in $K_n$ such that $f_n$ has exactly $k$ roots modulo $p.$ Observe, that

\begin{equation}\label{ineqroots}\sum_{p\leq x}\rho_n(p) \geq \sum_{k=1}^{n}\sum_{p\leq x \atop p\in C_k}k.\end{equation} 

 By Dedekind's theorem, up to finitely many exceptions, the primes $p$ such that $f_n$ has a root modulo $p$ correspond to the primes $p$ such that there exists a prime ideal $\mathfrak{p}$ in $\mathcal{O}_{K_n}$ above $p$ with inertial degree $f\left(\mathfrak{p}/p\right)=1.$ 
 
 Instead of working in the splitting field of $f_n$ and using Chebotarev's theorem as in \cite{Turkish}, we will directly count prime ideals in $K_n$ using prime ideal theorem. By the standard argument, the prime ideals of degree $>1$ will give negligible contribution. More precisely, we remark that counting prime ideals in $K_n$ of degree $1$ is equivalent to counting the rational primes $p$ with weight $k$ when $f_n$ has $k$ roots modulo $p.$ Thus, we have

 \begin{equation}\label{fundamental}\sum_{k=1}^{n}\sum_{p\leq x \atop p\in C_k}k = \sum_{N(\mathfrak{p})\leq x \atop f\left(\mathfrak{p}/p\right)=1} 1.\end{equation}

By the effective prime ideal theorem (see \cite[Theorem $5.33$]{IK})\footnote{We could equally apply effective version of Chebotarev theorem (\cite{effective}) for the trivial extension $K_n/K_n$.}, there exists an absolute constant $c>0$ such that for all $n\ge 1$
\begin{equation}\label{primeideal}\sum_{N(\mathfrak{p})\leq x \atop f\left(\mathfrak{p}/p\right)=1} 1 = \pi(x) + O\left(Li(x^{\beta_n}) + \frac{x}{\log x}\exp\left( -c\sqrt{\frac{\log x}{n^2}}\right)\right)\end{equation} where $\beta_n$ is the potential positive real zero of the Dedekind zeta function $\zeta_{K_n}$ and $$0<1-\beta_n \ll \frac{1}{\log d_{K_n}}\cdot$$ 

 We now restrict ourselves to the family of polynomials $\left\{f_{2n+1}(x), 1\leq n\leq N\right\}.$ 
 Recall that by the result of  Stark~\cite[Lemma $8$]{Stark}, we can control potential Siegel zeroes provided that the original extension does not contain any quadratic sub-extension. We do so here since $K_{2n+1}=\mathbb{Q}(\alpha)$ is of an odd degree. Latter yields the bound \begin{equation}\label{Stark}\beta_{2n+1} \leq 1-\frac{1}{4(2n+1)!\log \vert d_{K_{2n+1}}\vert}\cdot \end{equation}
 Using (\ref{ineqroots}) together with (\ref{primeideal}), we derive
\begin{align}\label{missing} \frac{1}{\pi(x)}\sum_{p\leq x} \left(p-V(0,p-1)\right)&\geq \sum_{n=1}^{N}\frac{1}{\pi(x)}\sum_{p\leq x}\rho_{2n+1}(p) \notag \\
& \geq \sum_{n=1}^{N}\frac{1}{\pi(x)}\sum_{k=1}^{2n+1}k \sum_{p\leq x \atop p\in C_k}1   \\ 
& \geq  N + O\left(\frac{1}{\pi(x)}\left\{\sum_{n=1}^{N} Li(x^{\beta_{2n+1}}) + \frac{x}{\log x}\exp\left( -c\sqrt{\frac{\log x}{(2n+1)^2}}\right)\right\}\right)\cdot\notag \end{align}

Hence, we have to choose parameter $N$ such that
\begin{equation}\label{Siegel} N \gg \sum_{n=1}^{N} \frac{1}{\pi(x)}\left\{Li(x^{\beta_{2n+1}})+ \frac{x}{\log x}\exp\left( -c\sqrt{\frac{\log x}{(2n+1)^2}}\right)\right\}.\end{equation}

The sum of the exponential terms in (\ref{Siegel}) satisfies  this as long as $N\ll \log^{1/2} x.$  

Since $K_{2n+1}$ is generated by the single root of $f_{2n+1},$ we can bound its discriminant by  the discriminant of the polynomial $f_{2n+1}$. Hence, denoting by $\alpha_i$ the roots of $f_{2n+1}$ we derive
 
 \begin{equation}\label{Discbound}d_{K_{2n+1}} \leq \prod_{i,j \atop i<j}^{2n+1} \vert\alpha_i-\alpha_j \vert^2 \ll n^{10n^2},\end{equation} where we used the bound $|\alpha_i-\alpha_j|\ll n$ which is proved in \cite[Lemma $2$]{Turkish}.

To bound the contribution coming from the potential Siegel zeroes, we apply the result of Stark (\ref{Stark}) together with the discriminant bound (\ref{Discbound}) to arrive at

\begin{equation}\label{loglog}\sum_{n=1}^{N} \frac{1}{\pi(x)} Li(x^{\beta_{2n+1}}) \ll \sum_{n=1}^{N} x^{-\frac{1}{n!\log(n^{n^2})}} \ll N x^{-\frac{1}{N^N N^2}}.\end{equation} 
Using the previous bound (\ref{loglog}) and standard computations, we deduce that inequality (\ref{Siegel}) is true as long as \begin{equation}\label{uncondbound}N \ll \frac{\log \log x}{\log \log \log x}\cdot\end{equation} 
We are left to note that the "bad" primes which do not satisfy Dedekind's theorem are exactly the primes dividing $\left[\mathcal{O}_{K_{2n+1}}:\mathbb{Z}\left[\alpha\right]\right].$ We have at most $\omega(2n+1))\ll \log n$ of such primes and, using (\ref{uncondbound}), their total contribution is at most 

$$\frac{1}{\pi(x)}\sum_{n\leq N}\sum_{p\leq x \atop \text{p 'bad' for } K_{2n+1}} \rho_n(p) \ll \frac{1}{\pi(x)}\sum_{n\leq N}n\log n \ll \frac{N^{2}\log N}{\pi(x)} =o (N).$$

\end{proof}

Assuming Generalized Riemann Hypothesis (GRH) the bound from Theorem ~\ref{average} can be significantly improved.  
  \begin{theorem}\label{GRH} 
Assume that GRH is true. Then, 
 
 $$\frac{1}{\pi(x)}\sum_{p\leq x} \left(p-V(0,p-1)\right) \gg \frac{x^{1/4}}{\log x}\cdot$$
 \end{theorem}
\begin{proof}

We consider as before the family of polynomials $f_n$ and the associated family of extensions $K_n$ of degree $n.$ Here, we do not need to restrict to odd $n$ because we use GRH instead of Stark's result.

Following the same lines as in the proof of Theorem \ref{average} and replacing the error term in the prime ideal theorem by the conditional one, we obtain

\begin{equation}\label{uncond}\sum_{p\leq x}\rho_n(p) \geq \pi(x) + O\left(x^{\frac{1}{2}}(\log d_{K_{n}}+n\log x)\right).\end{equation} Averaging over the family of polynomials $\left\{f_n(x), 1\leq n\leq N\right\}$ and performing the same computation as in (\ref{missing}), we arrive at
$$\frac{1}{\pi(x)}\sum_{p\leq x} \left(p-V(0,p-1)\right)\geq \sum_{n=1}^{N}\frac{1}{\pi(x)}\sum_{p\leq x}\rho_n(p) \gg N + ET.$$ Using the discriminant bound (\ref{Discbound}) we can bound error term by 

$$ET \ll \sum_{n=1}^{N}\left\{ x^{-\frac{1}{2}}\log x (\log (n^{n^2}) + n\log x)\right\} \ll\frac{\log^2 x}{\sqrt{x}} N^{3}.$$ Easy computation shows that the error term is negligible compared to $N$ provided $$N \ll  \frac{x^{1/4}}{\log x},$$ and the result follows. As in the proof of Theorem \ref{average}, we can easily deal with the additional restriction $p\nmid \left[\mathcal{O}_{K_{n}}:\mathbb{Z}\left[\alpha\right]\right].$ We bound the contribution of 'bad' primes in exactly the same way:
\begin{equation*}\frac{1}{\pi(x)}\sum_{n\leq N}\sum_{p\leq x \atop \text{p 'bad' for } K_n} \rho_n(p) \ll \frac{1}{\pi(x)}\sum_{n\leq N}n\log n \ll \frac{N^{2}\log N}{\pi(x)} =o (N). \end{equation*}

\end{proof}

 Theorem~\ref{GRH} directly implies:
\begin{corollary}\label{GRHinf}
Assume that GRH is true. There exists infinitely many primes $p$ such that 
\[p-V(0,p-1)\gg \frac{p^{1/4}}{\log p}\cdot \] 
\end{corollary}

 \begin{remark} Working in $K_n$ instead of the splitting field of $f_n$ allows us to get the bound on the discriminant exponentially smaller than the one used in~\cite{Turkish}. The main improvement then comes from the fact that counting prime ideals of degree $1$ in $K_n$ corresponds to counting primes with the appropriate weight  suitable for our problem. \end{remark}

  \section{Concluding remarks.  Erd\"os conjecture on average}
 
It would be interesting to prove Erd\"os conjecture for almost all primes $p.$  Indeed, if a prime $p$ satisfies Erd\"os conjecture, then at least two of the $f_n(x),\  n=1,\dots,p-1$ have a root modulo $p$. Chebotarev's theorem tells us that the density of primes $p$ such that $f_n(x)$ has no roots modulo $p$ is equal to the proportion of elements in $Gal(Spl(f_n))$ without fixed points.  The natural strategy would be to apply Chebotarev's theorem to the product  $f:=\displaystyle{\prod_{i}f_{n_i}}$ to control the density of primes failing Erd\"os conjecture. This amounts to understanding the proportion of elements in the  Galois group without fixed points. The following lemma helps us to do that:

 \begin{lemma}\label{fixed}
 Suppose that $G$ is a subgroup of $S_n$ acting on a set $X$ of $n$ elements. Then the proportion $\sigma_n$ of elements of $G$ that does not have any fixed point satisfies 
 \begin{equation}\label{Burnside}1-1/n! \leq \sigma_n \leq 1-1/n.\end{equation}
 \end{lemma}
 
 \begin{proof}
 We denote by $X^{\sigma}$ the number of elements of $X$ fixed by $\sigma \in G$ and $X\backslash G$ the number of orbits of the action of $G$ on $X$. By Burnside's lemma, we have that 
 
 $$|X\backslash G|=\frac{1}{|G|}\sum_{\sigma \in G} |X|^{\sigma}.$$ Hence
 
 $$1\leq |\{\sigma,\,\, X^{\sigma}\neq \varnothing\}|\frac{n}{|G|}$$ and the result follows. The other side of the inequality is easy because the identity elements fix points.

 \end{proof}
 
 If the splitting fields of $f_{n_i}(x)$ are disjoint we can apply Chebotarev's theorem to $f$ and bound the number of permutations without fixed points in $$ Gal(Spl(f)) \cong\prod_{i}Gal(Spl(f_{n_i}))  .$$ 
 
 Several computations provides support towards the following conjecture:  
 \begin{conjecture}\label{conj}
 Let $n_1\neq n_2$ be positive integers. Then 
 
 $$Spl(f_{n_1})\cap Spl(f_{n_2})=\mathbb{Q}.$$
 
 \end{conjecture}

 This conjecture together with Lemma~\ref{fixed} imply the Erd\"os conjecture on average. Indeed,  for each prime $p$ failing to satisfy  the aforementioned conjecture, each $f_n$ has at most $1$ root. Hence, the density of primes $S$ failing Erd\"os conjecture is

 \begin{align*}S\leq \sum_{n=1}^{N} (1-\sigma_n) \prod_{j\neq n}\sigma_j  =  \left(\sum_{n=1}^{N}\frac{(1-\sigma_n)}{\sigma_n}\right) \prod_{n=1}^{N}\sigma_n 
 \ll  \prod_{n=1}\left(1-\frac{1}{n}\right){\longrightarrow} \,0.
 \end{align*} where we used Lemma \ref{Burnside}.
 
We notice that proving Conjecture \ref{conj} even for a good proportion of $n$ would suffice.

\section*{Acknowledgements}

The authors would like to thank Andrew Granville and Igor Shparlinski for valuable comments.
During the preparation of this manuscript, O.K. was supported by the ISM doctoral grant and J. Armand Bombardier Foundation excellence award. M.M. was supported by a postdoctoral grant in CRM of Montreal under the supervision of Andrew Granville and Dimitris Koukoulopoulos.

\bibliography{factorials}
\bibliographystyle{alpha}

\end{document}